\documentclass[12pt]{amsart}

\usepackage{epsfig, graphicx} % Required for inserting images
\usepackage{amssymb}
\usepackage{amsmath}
\usepackage{amsthm}
\usepackage{mathtools}

\newcommand{\VAL}{\operatorname{VAL}}

\newtheorem{lem}{Lemma}
\newtheorem{thm}{Theorem}

\usepackage{hyperref}
\hypersetup{colorlinks=true,linkcolor=blue, linktoc=none, citecolor=blue,urlcolor=blue}

\begin{document}

\title{$PSL_2$ tropicalization and lines on surfaces}
\author[SHKOLNIKOV and PETROV]{Mikhail Shkolnikov, Peter Petrov}

\thanks{M.S. was supported by the Simons Foundation International grant no. 992227, IMI-BAS}

\address{\center{Institute of Mathematics and Informatics\hfill\break at the Bulgarian Academy of Sciences\hfill\break Akad. G. Bonchev St, Bl. 8, 1113 Sofia, Bulgaria}\vspace{8pt}}

\email{m.shkolnikov@math.bas.bg, pk5rov@gmail.com}

\begin{abstract}
    The paper is based on a talk given by the first author at the Gökova Geometry \& Topology conference in May 2024. The subject is an interplay between the ideas of tropical geometry and two-by-two matrices with an intention to explore new types of geometries. More concretely, the article gives a preliminary account for a non-abelian version of phase tropicalization for subvarieties of $PSL_2.$
    %No specific expert knowledge or familiarity with previous works is assumed from the reader, only some curiosity.
\end{abstract}


\maketitle

\section{Introduction} 
Tropical geometry is concerned with polyhedral spaces endowed with an integral affine structure and relations between them \cite{mikhalkin2009tropical}. Some authors \cite{maclagan2021introduction} choose to say that ``tropical geometry is a combinatorial shadow of algebraic geometry''. Indeed, in a number of setups, algebraic varieties (or their families) may cast a shadow that will have a tropical shape. Such a procedure is referred to as tropicalization. In the present article, we continue the study of similar type procedure initiated in \cite{mikhalkin2022non} and recently refined in \cite{shkolnikov2024introduction}.

We would like to remark that the quote above is not entirely accurate. First, there are complete tropical spaces that mimic the structure of non-algebraic complex manifolds, with most basic examples being generic tori and Hopf manifolds \cite{ruggiero2017tropical}. Second, there are symplectic aspects of the theory, allowing for example, to construct Lagrangian submanifolds starting with tropical curves \cite{mikhalkin2019examples} or tropical hypersurfaces \cite{matessi2021lagrangian}. Moreover, Gross-Siebert program puts tropical geometry at the core of the fundamental interplay between complex and symplectic geometries \cite{gross2011tropical}, namely the Mirror symmetry. 

On the other hand, it is certainly the straightest path that most introductions (such as the excellent ``Brief introduction to tropical geometry'' \cite{brugalle2015brief}) take, to imitate the spirit of classical algebraic geometry, where at least initially, one is interested in solutions of systems of polynomial equations, i.e. algebraic varieties. Speaking of solutions, we generally assume some arithmetic system, such as complex numbers, or some other ring or field. In tropical arithmetic, the numbers $\mathbb{T}$ are just real numbers together with $-\infty,$ which is the neutral element with respect to tropical addition, replaced now by the operation of taking the maximum. The role of multiplication is taken by the ordinary addition. These operations are degenerations of the classical ones by the procedure known as Litvinov-Maslov dequantization. 

Therefore, a tropical polynomial is a kind of convex piece-wise linear function. The notion of a root is also modified. Namely, we say that a tropical polynomial has a root at some point if the function is not smooth at it. This definition is best explained by redefining the operation of tropical addition to be multi-valued: $$``\alpha+\beta"=\begin{cases}\max(\alpha,\beta),\ \ \alpha\neq \beta\\
[-\infty,\alpha],\ \ \alpha=\beta.
\end{cases}
$$
With such an operation tropical numbers form what is known as a hyperfield \cite{krasner1956approximation,viro2010hyperfields} and the notion of zero locus of a polynomial is now the same as in classical algebraic geometry, i.e. the set of points where the polynomial attains the addition's neutral element. 

The motivation behind the above multi-valued refinement of the tropical addition can be explained as follows. Assume we have a real parameter $t\rightarrow\infty$ and take two families of complex numbers $A(t)$ and $B(t)$ such that the first asymptotic of each is known, i.e. $A(t)=at^\alpha+o(t^\alpha)$ and  $B(t)=bt^\beta+o(t^\beta).$ Then, the limit of $\log_{t}|A(t)+B(t)|$ is $\max(\alpha,\beta)$ if $\alpha\neq \beta,$ but can be anything less or equal otherwise. Taking such limits is the first instance of a tropicalization.

%The usual approach to tropical algebra is based on tropical semiring $\mathbb{R} \cup \{-\infty\}$ with max-plus operations, replacing the usual addition and multiplication. Yet another more general and seemingly more natural approach to tropical geometry mentioned briefly above uses hyperfields. The notion of hyperfield was proposed by Krasner (K), and in our days by Viro (V). Following it the algebraic structures (abelian groups, rings, fields) are re-defined with multi valued additive operation, and in the case of rings or fields the multiplication is an univalent operation, related to addition by slightly modified distributivity. In that way one get hypergroups, hyperrings, hyperfields, and moduli over hyperrings. This approach permits to follow the usual way of developing algebraic geometry, leading to introduction of tropical schemes over hyperfields (L). A tropical variety can be seen now as an algebraic set over the tropical hyperfield, defined as the set of points where the system of polynomials viewed as multivalued functions attains the element ${-\infty}$ in some of its solutions, beeing now a set in general. The element ${-\infty}$ is the tropical equivalent of 0 in the algebraic geometry as the neutral element of the tropical addition, that is taking maximum when working over semi-rings. 

Remarkably, a tropical polynomial of degree $d>0$ in one variable always has $d$ roots counted with multiplicities, i.e. the fundamental theorem of algebra holds in tropical arithmetic, so $\mathbb{T}$ is algebraically closed. If  a tropical Laurent polynomial in two variables is given, it defines a tropical curve $C$ in $(\mathbb{T}^*)^2=\mathbb{R}^2.$ From the geometric perspective, $C$ is a planar graph with straight edges of rational slope equipped with integer weights, satisfying the balancing condition at every vertex. Many classical facts hold for tropical curves. For example, Bernstein's theorem on the number of intersections of two planar curves in given deformation classes holds, provided we understand how to define the multiplicity of each intersection point. By understanding how tropical curves arise as limits of complex curves one is able to deduce analogous fact in algebraic geometry.

%One such instance is the celebrated Mikhalkin's correspondence theorem in the context of enumerative geometry \cite{mikhalkin2005enumerative}, which procured a huge boost of activity in the field. By counting tropical curves of a given degree and genus passing through an appropriate number of generic points one recovers the number of complex curves. In the tropical count, each curve comes with a multiplicity that is dictated by the combinatorial structure of its vertices. The multiplicity of a tropical curve tells how many curves in the classical count tropicalize to it.

%\textit{Here will be a paragraph citing some applications of tropical geometry in mathematics and beyond it. It was good here and there, and there, and there,  and there and also there and there. It is related to this and that and that. Outside mathematics tropical geometry has proven its value in computer science \cite{zhang2018tropical,maragos2021tropical}, physics \cite{borinsky2023tropical,borinsky2023tropical2,banerjee2023tropical,arkani2022feynman,tourkine2017tropical}, biology \cite{radulescu2020tropical,samal2016analysis,noel2012tropical} and economics \cite{tran2019product}, to list a few.} 

Before passing to the main geometric narrative of the paper, we would like to mention a few applications of tropical mathematics to other fields. 

Another important direction of investigation gives a general frame for developing tropical differential equations by tropicalizing the equations and their solutions. It is based on idempotent semirings and an idempotent version of differential algebra \cite{giansiracusa2024general}.

In combinatorics, parametrized versions of some classical algorithms for computing shortest-path trees have been expressed in tropical terms \cite{joswig2022parametric}. 

In statistics, there are models expressing  probability distributions as polynomial relations, turning to algebraic varieties, so they could be tropicalized (see \cite{pachter2005algebraic}). Tropicalized statistical models are fundamental for parametric inference, that is when analyzing the behavior of inference algorithms under the variation of model parameters. Inference is the evaluation of one or more coordinates of a single point on an algebraic variety, either over a field or over tropical semi-field.  

 For instance, the theoretical formulations of statistical mechanics appear to be naturally suited for tropical expression. For example taking the tropical limit of the Schrödinger equation leads to the classical Hamilton–Jacobi equation, the procedure being known in the mathematical literature as Maslov dequantization (see\cite{litvinov1996idempotent,kolokoltsov2013idempotent}), which may be reminiscent of the WKB approximation.

In biology, an application of tropical geometry leads to simplifying complex biological models. Its methods exploit a feature of biological systems called multiscaleness \cite{gorban2008dynamic,samal2016geometric}.

Another area of applications to mention is in neural networks. They are based on the hope that in particular it would shed light on the working of deep neural networks \cite{zhang2018tropical}.

Many classical results such as the fundamental theorem of algebra, B\'ezout theorem and Pascal’s theorem,  have been re-proved in tropical context \cite{grigg2007elementary,hoffman2020pascal,rimmasch2008complete}. Others, like the theorem of 27 lines on a smooth cubic surface, has been modified appropriately when proved in tropical geometry \cite{panizzut2022tropical}.  In a number of  cases the tropical approach gives another way to obtain results but in a supposedly simpler way, a consequence of the piecewise linear structure of tropical varieties. An example is given by the celebrated Mikhalkin’s correspondence theorem in the context of enumerative geometry \cite{mikhalkin2005enumerative}, which procured a huge boost of activity in the field. By counting tropical curves of a given degree and genus passing through an appropriate number of generic points one recovers the number of complex curves, that has been asked initially. In the tropical count, each curve comes with a multiplicity that is dictated by the combinatorial structure of its vertices. The multiplicity of a tropical curve tells how many curves in the classical count tropicalize to it. This approach has been used successfully for other problems from enumerative geometry, like the theorem of Caporaso-Harris \cite{caporaso1998counting,itenberg2004logarithmic,gathmann2008kontsevich}, giving at the same time an alternative algorithmic methods for eliminating variables \cite{sturmfels2008tropical}.

An important feature of the planar case is that every planar tropical curve is realizable as a tropicalization. This fact extends to tropical hypersurfaces in $\mathbb{R}^n.$ However, already for tropical curves in $\mathbb{R}^3$ this is not true. For example, in classical geometry, every spacial cubic curve of genus one belongs to some plane, but there is a tropical spacial cubic of genus one that doesn't belong to any tropical plane. The question of realizability is very subtle, a tropical surface and a tropical curve on it can be realizable separately, but not as a pair as was shown in \cite{brugalle2015obstructions}. 

The two most conventional versions of tropicalization rely on the concept of amoeba, where the name comes from the famous book \cite{gelfand1994discriminants}. In a more geometric version, one considers the scaling limit of amoebas of subvarieties in the complex algebraic torus \cite{mikhalkin2004amoebas}. For an algebraic subvariety $V\subset(\mathbb{C}^*)^n$ its amoeba is $\operatorname{Log}(V)\subset\mathbb{R}^n,$ where $\operatorname{Log}\colon(\mathbb{C}^*)^n\rightarrow\mathbb{R}^n$ sends $(z_1,\dots,z_n)$ to $(\log|z_1|,\dots,\log |z_n|).$ Amoebas have many remarkable properties that we are not going to discuss here. By replacing $\mathbb{C}$ in the previous sentence with a field $\mathbb{K}$ endowed with an ultra-norm $|\cdot|_{\mathbb{K}^*}$ (where $-\log\circ|\cdot|_{\mathbb{K}^*}$ is a valuation) we get the notion of a non-Archimedean amoeba, and its closure is called a tropicalization of the subvariety in this approach \cite{amini2014tropical}.

In \cite{mikhalkin2022non} a modified setup was considered. Namely, the map $\operatorname{Log}$ was interpreted as taking the quotient of the algebraic torus by its maximal compact subgroup. This interpretation allows a replacement of the torus by some other algebraic group, for example, $SL_2(\mathbb{C})$ or its central quotient $PSL_2(\mathbb{C}),$ for which the corresponding amoebas were studied, as well as their tropical limits in the case of families of curves. 

What shall be discussed in the next sections, is a modest extension of this work in the case of $PSL_2(\mathbb{C})$. We stick with this particular choice of the group since its natural compactification is the projectivization of the space of complex two-by-two matrices, identified with $\mathbb{C}P^3.$ It is slightly more convenient to us than the compactification of $SL_2(\mathbb{C}),$ which is a smooth projective quadric three-fold. We are not yet attempting to present an ultimate description of $PSL_2$ tropical geometry, but rather argue that it requires some decorations in order to capture even the most basic phenomena in the interplay of spatial curves and surfaces. 

\section{$PSL_2$ tropicalization}
A maximal compact subgroup of the complex algebraic group $PSL_2(\mathbb{C})$ is $PSU(2),$ the quotient of the group $SU(2)$ of unitary matrices with determinant $1$ by the subgroup $\{I_2,-I_2\},$ where $I_2$ denotes the two-by-two identity matrix. Topologically, $SU(2)$ is a three-dimensional sphere. One way to see it is via the identification of $SU(2)$ with the group of quaternions of norm one. More explicitly, $$SU(2)=\{\begin{pmatrix} z&w\\ -\overline{w}&\overline{z}\end{pmatrix}
\vert z,w\in\mathbb{C},|z|^2+|w|^2=1\}.$$
The quotient map $SU(2)\rightarrow PSU(2)$ is a degree two covering, thus $PSU(2)$ is topologically a three-dimensional projective space. 

In fact, it is also true geometrically in the following sense. As it was mentioned, $PSL_2(\mathbb{C})$ is naturally compactified to $\mathbb{C}P^3,$ the projectivization of the space of complex two-by-two matrices.  Moreover, there exists a real structure on it such that $PSU(2)$ is the corresponding set of real points. Let us construct this real structure explicitly. First, note that the operation of group-theoretical invertion on $PSL_2(\mathbb{C})$ extends to a linear involution on $\mathbb{C}P^3$ given by $[B]_{\mathbb{C}^*}\mapsto[B^{\pmb{c}}]_{\mathbb{C}^*},$ where 
$\begin{pmatrix}
    a&b\\c&d
\end{pmatrix}^{\pmb{c}}=\begin{pmatrix}
    d&-b\\-c&a
\end{pmatrix}$
and $[B]_{\mathbb{C}^*}\in\mathbb{C}P^3$ denotes the class of a nonzero matrix up to a rescaling by a non-zero complex number. Then, composing it with the Hermitian conjugation, $[B]_{\mathbb{C}^*}\mapsto[(B^{\pmb{c}})^*]_{\mathbb{C}^*}$ gives an anti-holomorphic involution on $\mathbb{C}P^3$ with the set of fixed points being $PSU(2).$ Observe also that the quadric surface $Q(\mathbb{C})=\mathbb{C}P^3\backslash PSL_2(\mathbb{C})$ given by the homogeneous equation $\det(B)=0$ is invariant under this involution but doesn't intersect $PSU(2).$ In other words, $Q(\mathbb{C})$ is a real quadric with no real points, such quadrics are also called imaginary. 

The analog of the classical logarithmic projection in the theory of amoebas is the quotient map $\varkappa\colon PSL_2(\mathbb{C})\rightarrow PSL_2(\mathbb{C})\slash PSU(2).$ We call $\varkappa$ the hyperbolic amoeba map since the target is naturally a hyperbolic 3-space $\mathbb{H}^3,$ which will now be explained. Note that there is a polar decomposition in the space of complex matrices, where a matrix $A$ is decomposed into a product $PU$ with $U$ being unitary (the angular part) and $P$ being Hermitian (the radial part). If, furthermore, $A$ is non-degenerate we may assume that $P$ is positive definite, and in such case the decomposition is unique. The map $\varkappa$ is seen as forgetting the unitary factor and keeping only the hermitian one. 

Thus, our model for the hyperbolic three-space is $$\mathbb{H}^3=\{P\in \operatorname{Mat}_{2\times 2}(\mathbb{C})|P=P^*,P>0, \det(P)=1\}.$$
This presentation essentially coincides with the Minkowski model, since the space of two-by-two Hermitian matrices is a real four-dimensional vector space, and the determinant induces on it the quadratic form of signature $(1,3).$ The set $\det(P)=1$ is a two-sheeted hyperboloid and the condition $P>0$ of being positively definite chooses one of the two sheets $\mathbb{H}^3$. Note that $-\det$ restricted to $\mathbb{H}^3$ defines a Riemannian metric. More details about this construction can be found in \cite{thurston1997three}. 

With this metric, our $\mathbb{H}^3$ is a homogeneous space since $PSL_2(\mathbb{C})$ acts on it isometrically and transitively (in fact, it is the group of all orientation-preserving isometries). Explicitly the action is given by $A(P)=APA^*,$ for $A\in SL_2(\mathbb{C})$ which clearly factors through $PSL_2(\mathbb{C}).$ Note that our map $\varkappa$ may be now seen as $\varkappa([A]_{\mathbb{C}^*})=A(I_2),$ where $\det(A)=1$ is assumed. Such a map is equivalent to taking the quotient by the stabilizer of $I_2,$ which consists of matrices $A$ such that $AA^*=I_2$ -- the condition of being unitary. As a byproduct, we just derived the classical isomorphism $PSU(2)\cong SO(3),$ where the latter is a group of rotations around a point in three-space.

Strictly speaking, the Hermitian part in the polar decomposition $A=PU$ is not exactly $\varkappa(A)=AA^*,$ but rather its square root since $AA^*=PUU^*P=P^2.$ This difference is minor, however, since taking power $h>0$ of matrices in our model of $\mathbb{H}^3$ defines a homothety $R_h(P)=P^h$ centered in $I_2$ and with scaling factor $h.$ One way to see it is by observing that the distance from $P\in\mathbb{H}^3$ to $I_2$ is computed as the absolute value of the logarithm of an eigenvalue of $P.$  

For completeness, we mention the spherical coamoeba map $$\varkappa^\circ\colon PSL_2(\mathbb{C})\rightarrow PSU(2)$$ which assigns the unitary part of the polar decomposition and forgets the Hermitian one. It will play an important role in describing $PSL_2$ phase tropicalization. The explicit formula of $\varkappa^\circ$ is the following.

\begin{lem}
    Let $[A]_{\mathbb{C}^*}\in PSL_2(\mathbb{C})$ such that $\det(A)=1.$ The value of the coamoeba map is expressed as $\varkappa^{\circ}([A]_{\mathbb{C}^*})=[A+(A^{\pmb{c}})^*]_{\mathbb{R}^*}\in \mathbb{R}P^3=PSU(2).$
    \label{lem_sphcoam}
\end{lem}

\begin{proof} Let $A=PU$ be the polar decomposition. Then, $P=\sqrt{AA^*}$ and $U=P^{-1}A.$ We rewrite the square root expression using Cayley-Hamilton theorem that asserts the annihilation of any matrix by its characteristic polynomial. Explicitly, for a two-by-two matrix $B$ we have $\det(B)I_2-\operatorname{tr}(B)B+B^2=0.$ Substituting $B=\sqrt{AA^*}$ we get $[P]_{\mathbb{R}^*}=[I_2+AA^*]_{\mathbb{R}^*}.$ Then, $$\varkappa^\circ([A]_{\mathbb{C}})=[(I_2+AA^*)^{\pmb{c}}A]_{\mathbb{R}^*}=[A+(A^{\pmb{c}})^*A^{\pmb{c}}A]_{\mathbb{R}^*}=[A+(A^{\pmb{c}})^*]_{\mathbb{R}^*}.\vspace{-20pt}$$
\end{proof}

Here is the geometric version of $PSL_2$ tropicalization of hyperbolic amoebas of curves as it is described in \cite{mikhalkin2022non}. 
A hyperbolic amoeba of a subvariety $V$ in $PSL_2(\mathbb{C})$ is $\varkappa(V).$
A scaling sequence is $\{t_\alpha\}_{\{\alpha\in\pmb{A}\}},$ an unbounded subset of real numbers with $t_\alpha>1.$ Then, for a family of degree $d$ reduced irreducible curves $V_\alpha\subset PSL_2{\mathbb{C}}$ there exists a subsequence $V_\beta$ such that $R_{t_\beta}(\varkappa(V_\beta))\subset\mathbb{H}^3$ converges in the Hausdorff sense to an $\mathbb{H}^3$-spherical complex associated with a degree $d$ $\mathbb{H}^3$-floor diagram. The formal definitions of the latter notions are not relevant for the present article so we are not repeating it here. Vaguely speaking, the tropical limits of hyperbolic amoebas of curves are unions of spheres centered in $I_2$ (floors) and segments of geodesics passing through $I_2$ (elevators). In the particular case of $d=1,$ there are two options: either there is a single elevator and no floors, or there is a single floor and two elevators. 

In \cite{mikhalkin2022non}, the case of $PSL_2$ tropical limits of surfaces was not covered. This is mostly due to the reason that the result was not expected to be informative. Indeed, even before the limit, if $V\subset PSL_2(\mathbb{C})$ is an odd degree surface, then $\varkappa(V)=\mathbb{H}^3,$ so nothing can be extracted from it. The even degree case is slightly less trivial, i.e. $\mathbb{H}^3\backslash\varkappa(V)$ can be non-empty but it has at most one connected component which is convex. In a separate article, we will show (in the geometric setup of scaling sequences) that a tropical limit of hyperbolic amoebas of even degree surfaces is always a complement to a geometric open ball centered in $I_2.$ 

The moral of all these is that bare $PSL_2$ tropicalization is not good for observing the geometry of surfaces in $\mathbb{P}^3.$ For example, on a generic quintic surface there are no lines. But in the hyperbolic-tropical picture, we see infinitely many of their images! The aim of the present article is to show that this affliction has the potential to be cured. We will show that by restoring the phase in $PSL_2$ tropicalization at least such type of a basic fact becomes visible.

\section{Phase tropicalization}
The idea of phase tropicalization is that we consider the scaling limit of a family of varieties inside their ambient space, i.e. without forgetting the phase. In the setup of the previous section, it means lifting the $h$-homotheties $R_h:\mathbb{H}^3\rightarrow\mathbb{H}^3$ to the $h$-parametric family of diffeomorphisms $\tilde{R}_h:PSL_2(\mathbb{C})\rightarrow PSL_2(\mathbb{C})$ so that $R_h\circ\varkappa=\varkappa\circ\tilde{R}_h$ (similarly, $\varkappa^\circ\circ\tilde{R}_h=\varkappa^\circ$ holds for the coamoeba map), and considering limits of $\tilde{R}_{t_\alpha}(V_\alpha)$ inside $PSL_2(\mathbb{C}).$ 
The lift can be performed using the polar decomposition $\tilde{R}_h(PU)=P^hU,$ and it continuously extends to $Q(\mathbb{C})$ as an identity map. 

Although it seems to be a triviality at first glance, phase tropicalization of a family of points opens the door for the possibility of computing (or, at least restricting) phase tropical limits of higher dimensional varieties. For example, in the classical situation of the complex algebraic torus consider a family of numbers $z(t)\in\mathbb{C}^*$ parametrized by real $t\rightarrow\infty.$ Assume that the first asymptotic is known, i.e. $z(t)=ct^\alpha+o(t^\alpha)$ for some $c\in\mathbb{C}^*$ and $\alpha\in\mathbb{R}.$ Then $\log_t|z(t)|\rightarrow \alpha,$ which is the usual tropical limit, we may call the assignment $z(t)\mapsto \alpha$ an ordinary valuation. The phase tropical limit arises if we consider the same type of rescaling but inside $\mathbb{C}^*$, i.e. $$\frac{z(t)}{|z(t)|}|z(t)|^{(\log(t))^{-1}}\xrightarrow[t\rightarrow\infty]{} \exp(\alpha +i\operatorname{Arg}(c)).$$   
The assignment to $z(t)$ of the expression on the right is the phase valuation. Applied coordinate-wise, to each point of a subvariety of a higher dimensional algebraic torus, this gives a classical phase-tropical limit.

In the previous sentence, we meant implicitly that our subvariety that we wish to tropicalize is now defined over a field $\mathbb{K}$ whose elements depend on a parameter $t.$ There are various choices for such a field and it is not our intention to discuss all the technical issues specific to different choices. The intuition, however, is as follows. When a subvariety of a torus is defined by a system of equations with coefficients depending on the parameter $t$ it may be treated either as a family of varieties $V_t$ depending on $t$ or as a single variety $V$ over $\mathbb{K}.$ In the former interpretation, we may consider the (phase) tropicalization of this family as a limit in $t\rightarrow\infty$, and in the latter one, as a closure of the image of a single variety $V$ under the coordinate-wise (phase) valuation. It is generally assumed that the two approaches agree, however, we are not able to suggest a concrete reference where the fact is proven in full generality. 

Coming back to matrices, assume that $[A(t)]_{\mathbb{C}^*}\in PSL_2(\mathbb{C})$ and $A(t)=t^\alpha B+o(t^\alpha)$ as $t\rightarrow\infty.$ Note that if $\det(A(t))=1$ then $\alpha\geq 0.$ The classical case suggests what should be the scaling $h$ in $\tilde{R}_{h}$ for $\tilde{R}_{h(t)}([A(t)]_{\mathbb{C}^*})$ to converge, namely we take $h(t)=(\log(t))^{-1}.$

\begin{lem}
Consider $A(t)=t^\alpha B+o(t^\alpha)$ as $t\rightarrow\infty.$ If $\det(A(t))=1,$ then $\tilde{R}_{(\log(t))^{-1}}([A(t)]_{\mathbb{C}^*})$ converges to $[e^\alpha B+e^{-\alpha} (B^{\pmb{c}})^*]_{\mathbb{C}^*}\in PSL_2(\mathbb{C}).$ 
\label{lem_tropofpoint}
\end{lem}

\begin{proof}
Note that if $\alpha=0$ and so $\det(B)=1,$ then the limit is just the value of the spherical coamoeba map $\varkappa^\circ$ on $B$ (see Lemma \ref{lem_sphcoam}). Otherwise, if $\alpha>0$ and so $\det(B)=0,$ we find separately the limits of the unitary and the Hermitian parts in the polar decomposition for $\tilde{R}_{(\log(t))^{-1}}([A(t)]_{\mathbb{C}^*})$. 
The first one is simply $[B+(B^{\pmb{c}})^*]_{\mathbb{R}^*}$ and the second one is $[e^{-\alpha}\operatorname{tr}(BB^*)I_2+(e^{\alpha}-e^{-\alpha})BB^*]_{\mathbb{C}^*}.$ As an intermediate step for the latter, find the direction at infinity in which $A(t)(A(t))^*$ goes as well as the asymptotic of the distance to $I_2$ using the relation between the trace to the eigenvalues, and the eigenvalues to the distance. Combining the parts together and using the identities $BB^{\pmb{c}}=0$ and $BB^*B=\operatorname{tr}(BB^*)B$ for two-by-two matrices of rank one, we arrive at the desired expression.
\end{proof}

We would like to use the formula derived  in Lemma \ref{lem_tropofpoint} as a definition for the new valuation $\VAL:PSL_2(\mathbb{K})\rightarrow PSL_2(\mathbb{C}).$ 
Given $A(t)=t^\alpha B+o(t^\alpha)$ with $\det(A(t))=1,$ let $$\VAL([A(t)]_{\mathbb{K}^*})=[e^\alpha B+e^{-\alpha} (B^{\pmb{c}})^*]_{\mathbb{C}^*},$$ be the $PSL_2$ phase valuation. Note also that if $\det(A(t))=0$ then $\tilde{R}_{(\log(t))^{-1}}([A(t)]_{\mathbb{C}^*})$ converges to $[B]_{\mathbb{C}^*}\in Q(\mathbb{C}),$ thus we may extend $\VAL$ to the map from $\mathbb{K}P^3$ to $\mathbb{C}P^3.$

The formula for $\VAL$ of Lemma \ref{lem_tropofpoint}, although explicit, is not that convenient to work with. We will use a different perspective on the space $PSL_2(\mathbb{C})$ as well as its compactification $\mathbb{C}P^3$. In order to do it, let's recall the classical one-to-one correspondence between oriented lines in $\mathbb{R}P^3$ and complex points of an imaginary quadric surface. Namely, for every point $p\in Q(\mathbb{C})$ consider its conjugate point $q$ and trace a complex line $l(\mathbb{C})$ through them. By construction, this line is invariant under the real structure, so it makes sense to look at $l(\mathbb{R})\subset\mathbb{R}P^3$ which cuts $l(\mathbb{C})\cong\mathbb{C}P^1$ into two discs. Moreover, since $l(\mathbb{C})$ is canonically oriented being a complex manifold, these discs define (opposite) orientations on their common boundary. We pick the orientation on $l(\mathbb{R})$ induced by the disc containing the point $p$ we started with.

As a byproduct of this construction, we define the circle bundle $\mathcal{S}$ on $Q(\mathbb{C})$ where a fiber over a point $p$ is the corresponding real oriented line. We need to write the fiber a bit more explicitly, namely for a point $p=[B]_{\mathbb{C}^*}\in Q(\mathbb{C}),$ the fiber $\mathcal{S}_p$ is $[cB+\overline{c} (B^{\pmb{c}})^*]_{\mathbb{R}^*},$ where $c$ runs through the set of all nonzero complex numbers. This set is canonically identified with classes $[cB]_{\mathbb{R}^*}.$ Therefore, we may write the total space of the circle bundle $\mathcal{S}$ simply as the set of all classes $[B]_{\mathbb{R}^*},$ where $B$ runs over all nonzero complex rank one matrices. The projection $\mathcal{S}\rightarrow Q(\mathbb{C})$ is just $[B]_{\mathbb{R}^*}\mapsto[B]_{\mathbb{C}^*}.$ 
Then, there is a diffeomorphism $(0,\infty)\times \mathcal{S}\rightarrow PSL_2(\mathbb{C})\backslash PSU(2)$ given by $(\alpha,[B]_{\mathbb{R}^*})\mapsto [e^\alpha B+e^{-\alpha} (B^{\pmb{c}})^*]_{\mathbb{C}^*}.$ We complete this presentation by gluing spaces $(0,\infty)\times \mathcal{S},$ $\{0\}\times PSU(2)$ and $\{\infty\}\times Q(\mathbb{C}).$

We call the result the {\em cone picture} of $\mathbb{C}P^3,$ where at the base at infinity there is the imaginary quadric, at the tip there is the real projective space, and in between is the cylinder of the total space of the circle bundle. Note that the fibers of the circle bundle are glued nicely both to the tip and the base. In the first case, the fiber converges to the corresponding line in the real projective space and, in the second case, it contracts to a point on the base. In this picture, the expression of $\VAL$ is particularly nice. Namely, for $A(t)=t^\alpha B+o(t^\alpha),$ $$\VAL([A(t)]_{\mathbb{C}^*})=\begin{cases}
(0,\varkappa^\circ(B)),\ \ \det(A(t))=1, \det(B)=1,\ \alpha=0;\\
(\alpha,[B]_{\mathbb{R}^*}),\ \ \; \det(A(t))=1, \det(B)=0,\ \alpha>0;\\
(\infty,[B]_{\mathbb{C}^*}),\ \ \det(A(t))=0.
\end{cases}
$$
 
\section{The case of lines}
The goal of this section is to compute the image under $\VAL$ for all lines. Starting from now we will be working over the field $\mathbb{K}$ of Puiseux series with complex coefficients and variable $t\rightarrow\infty.$ To make the formulas a bit less cumbersome, in the statements of Theorems below we have chosen to suppress writing the intersections of the intervals with the set of rational numbers, which are the possible powers of the type of series we are considering. For the sake of simplifying notations, in what follows we suppress the dependence on $t$ for elements of $\mathbb{K}.$

We start with a line $L$ in $\mathbb{K}P^3$ intersecting the quadric $Q(\mathbb{K})$ at a single point $[A]_{\mathbb{K}^*}.$ Take any other point $[B]_{\mathbb{K}^*}$ on $L.$ Then, $L\backslash\{[A]_{\mathbb{K}^*}\}$ is parametrized as $[zA+B]_{\mathbb{K}^*}$ by $z\in\mathbb{K}.$ Note that $\det(B)\neq 0,$ thus we may assume that $\det(B)=1$ by choosing a different representative in $[B]_{\mathbb{K}^*}.$ This assumption implies also that $\det(Az+B)=1$ for all $z\in\mathbb{K}$. This follows from the fact that $\det(zA+B)$ is, in general, a quadratic polynomial in $z$ with the leading term being $\det(A)z^2,$ which vanishes since $A$ is on the quadric. Thus, $\det(zA+B)$ is a linear polynomial. Since $L$ intersects the quadric at a single point, the polynomial has no roots and, therefore, is constant.

Now write $A$ as $A_0t^\alpha+o(t^\alpha)$ and $B$ as $B_0t^\beta+o(t^\beta),$ where $A_0$ and $B_0$ are nonzero matrices with complex entries and $\beta\geq 0$ since $\det(B)=1$. Take $z=ct^\gamma+o(t^\gamma),$ where $c$ is a nonzero complex number. To find the top-order term in $Az+B,$ look at the tropical polynomial $\gamma\mapsto\max(\alpha+\gamma,\beta),$ which has a single breaking point at $\gamma_0=\beta-\alpha\geq 0$ and the minimal value equal to $\beta.$ Thus, if $\gamma>\gamma_0$ the leading term is $A_0t^{\gamma+\alpha},$ if $\gamma<\gamma_0$ it is $Bt^\beta$ and if $\gamma=\gamma_0$ these two expressions compete, which might lead to a cancellation, and a lower leading term as a result.

Let us first resolve the latter situation. The cancellation may happen only if $A_0$ and $B_0$ are proportional, i.e. if $B_0=bA_0$ for some $b\in\mathbb{C}.$ Take a different parametrization of $L$ by substituting $z=w-bt^{\gamma_0}$ which results in $Az+B=Aw+B',$ where $B'=B-bt^{\gamma_0}A$ has a strictly lower order highest term than $B.$ Therefore, assuming that $B$ in $zA+B$ has the smallest top order term among all such parameterizations of the same line makes this situation impossible, and we may discard this possibility.

Now, assume $\beta>0$ in which case $\VAL(L)$ doesn't have points over the tip of the cone $\widehat Q.$ If $\gamma>\gamma_0,$ the term $cA_0t^{\gamma+\alpha}$ dominates in $Az+B.$ In the cone picture, this gives a ray with base $[A_0]_{\mathbb{C}^*}$ and going between the heights $\gamma_0$ and $\infty.$ Since $c\in\mathbb{C}^*$ may vary arbitrarily we get the whole circle in $\mathcal{S}$ over each interior point of the ray in the image of $\VAL$.

At $\gamma=\gamma_0,$ the top-order term is $(cA_0+B_0)t^\beta.$ Note that $\det(cA_0+B_0) = 0$ since we assumed that $\beta>0.$ Thus $c\in\mathbb{C}^*\mapsto [cA_0+B_0]_{\mathbb{C}^*}$ parametrizes a line $l$ on the quadric $Q(\mathbb{C})$ with two points $[A_0]_{\mathbb{C}^*}$ and $[B_0]_{\mathbb{C}^*}$ missing. Note, however, that for the purpose of computing $\VAL,$ we are rather interested in $[cA_0+B_0]_{\mathbb{R}^*}$ which defines a section of $\mathcal{S}$ restricted to $l.$ It is not possible to continuously extend this section at $[A_0]_{\mathbb{C}^*},$ the base of the above-mentioned ray, but at the other missing point the section does extend by $[B_0]_{\mathbb{R}^*},$ which corresponds precisely to the image of $\VAL$ for all the points on $L$ parametrized by $z$ with $\gamma<\gamma_0.$ 

To complete the description of $\VAL(L)$ for $L$ tangent to $Q(\mathbb{K}),$ we are left with the case $\beta=0.$ Again, the ray part corresponding to $\gamma>0$ and domination of $cAt^{\gamma+\alpha}$ is exactly same as before. What is different is that now we have points over the tip of the cone corresponding to $\gamma\leq -\alpha.$ Under $\VAL,$ all the points of $L$  parametrized by $z$ with $\gamma<-\alpha$ will be mapped to a single value $\varkappa^\circ[B_0]_{\mathbb{C}^*}=[B_0+(B_0^*)^{-1}]_{\mathbb{R}^*}.$ Finally, if $\gamma=-\alpha$ we get the spherical coamoeba of a complex line parametrized by $c\mapsto [cA_0+B_0]_{\mathbb{C}^*}.$ Since it will not play any role in the discussion below, we simply state that this coamoeba is equal to the plane spanned by the point $\varkappa^\circ[B_0]_{\mathbb{C}^*}$ and the real line corresponding to $[A_0]_{\mathbb{C}^*}\in Q(\mathbb{C})$ with the line removed (here we also implicitly claim that the point cannot belong to the line). To summarize, we have the following.

\begin{thm}
Let $L$ be a line in $\mathbb{K}P^3$ intersecting $Q(\mathbb{K})$ in a single point. Then, $\VAL(L)$ is either \begin{equation}\{r\}\times\sigma(l\backslash\{p\}) \cup (r,\infty)\times\mathcal{S}|_{\{p\}}\cup\{\infty\}\times\{p\},\label{eq_tangent_r}\end{equation} where $r>0,$ $l$ is a line on $Q(\mathbb{C}),$ $p\in l$ and $\sigma$ is a section of $\mathcal{S}|_{l\backslash\{p\}},$ or
\begin{equation}\{0\}\times\varkappa^\circ(l\backslash \{p\})\cup (0,\infty)\times\mathcal{S}|_{\{p\}}\cup\{\infty\}\times\{p\},\label{eq_tangent_0}\end{equation}
where $l$ is a line on $\mathbb{C}P^3$ tangent to $Q(\mathbb{C})$ at the point $p.$  
\label{thm_tangentlines}
\end{thm}

Now we proceed with the case of a line $L$ intersecting the quadric $Q(\mathbb{K})$ at two points. Denote such a pair of points by $[A]_{\mathbb{K}^*}$ and $[B]_{\mathbb{K}^*},$ where $A$ and $B$ are nonzero matrices with entries in $\mathbb{K}$ having vanishing determinant. Then, we can parametrize the line $L$ by $w\mapsto[Aw+B]_{\mathbb{K}^*}.$ Again, the vanishing of $\det(A)$ implies that $\det(Aw+B)$ is a linear polynomial. Its only root is at $w=0$ since $\det(B)=0.$ Therefore $\det(Aw+B)=uw$ for some  $u\in\mathbb{K}^*.$ In order to compute $\VAL,$ we normalize $Aw+B$ by dividing it by $\sqrt{uw}.$ Thus, after appropriately rescaling $A$ and $B,$ we may assume that the line $L$ without the two points on $Q(\mathbb{K})$ is doubly covered by $\mathbb{K}^*$ with the map given by $z\mapsto[Az+Bz^{-1}]_{\mathbb{K}^*}.$

Write $A=A_0t^{\alpha}+o(t^\alpha),$ $B=B_0t^{\beta}+o(t^\beta)$ and $z=ct^\gamma+o(t^\gamma),$ where $A_0$ and $B_0$ are rank one complex matrices and $c\in\mathbb{C}^*.$ To find the leading power of $Az+Bz^{-1},$ we look at the tropical polynomial $\gamma\mapsto\max(\alpha+\gamma,\beta-\gamma).$ Its breaking point, i.e. the tropical root, is at $\gamma_0=\frac{\beta-\alpha}{2},$ with the value at it being $r=\frac{\beta+\alpha}{2}.$ Note that since $\det(Az+Bz^{-1})=1,$ and $r$ bounds from above the order of the leading term in $Az+Bz^{-1}$ at $\gamma=\gamma_0,$ it must be non-negative.

If $r=0,$ then for $\gamma=\gamma_0$ no cancellation is possible, since it would further lower the order of $Az+Bz^{-1}$, i.e. $A_0$ and $B_0$ are not proportional and thus $p=[A_0]_{\mathbb{C}^*}$ and $q=[B_0]_{\mathbb{C}^*}$ are distinct points on $Q(\mathbb{C}).$ Moreover, $1=\det(Az+Bz^{-1})=\det(cA_0+c^{-1}B_0)+o(1)$ for all $c\in\mathbb{C}^*$ implying that the line $l$ passing through $p$ and $q$ doesn't lie on $Q(\mathbb{C}).$ The value of $\VAL$ for $\gamma=\gamma_0$ is $\varkappa^\circ(l\backslash\{p,q\}).$ For $\gamma>\gamma_0$ the leading term in the parametrization is $cA_0t^{\alpha+\gamma},$ which gives a ray going all the way from the tip to $p$ at the base at infinity, and the whole circle fiber of $\mathcal{S}.$ Same goes for the case $\gamma<\gamma_0$ with $p$ being replaced by $q.$

If $r>0,$ cancellations are in principle possible, but only when $A_0$ and $B_0$ are proportional. Assume for now the contrary, i.e. that they represent two distinct points $p$ and $q$ on $Q(\mathbb{C})$. Note that the line $l$ passing through these two points now lies on the quadric since $\det(cA_0+c^{-1}B_0)=0$ for all $c\in\mathbb{C}^*$ because $1=\det(Az+Bz^{-1})=\det(cA_0+c^{-1}B_0)t^{2r}+o(t^{2r})$ and $2r>0.$ Thus, at $\gamma=\gamma_0$ and by varying $c$ we get a continuous section $\sigma$ of $\mathcal{S}|_{l\backslash\{p,q\}}$ as the image of $\VAL.$ Again, for $\gamma\neq\gamma_0,$ we get two rays going from $p$ and $q$ until the height $r$ with the whole fiber of $\mathcal{S}$ over each point. 

We are left to study the case when the cancellation occurs. As mentioned, this may happen only if $r>0.$ By rescaling, we may assume that it happened at $z=1,$ in particular the series expansion of matrices $A$ and $B$ starts as $A=A_0t^r+A_1t^{\alpha_1}+o(t^{\alpha_1})$ $B=-A_0t^r+B_1t^{\beta_1}+o(t^{\beta_1}),$ where $r>\alpha_1,\beta_1.$ Write also $z=1+ut^\delta+o(t^\delta),$ where $\delta<0,$ and $z^{-1}=1-ut^\delta+o(t^\delta).$ We have
$$Az+Bz^{-1}=2A_0ut^{r+\delta}+A_1t^{\alpha_1}+B_1t^{\beta_1}+\text{smaller order terms.}$$
Note that since the determinant of this expression is always $1,$ by substituting $u=0$ we get $r_1:=\max(\alpha_1,\beta_1)\geq 0.$ If $r_1=0,$ then we get the same type of expression as (\ref{eq_tangent_0}) in Theorem \ref{thm_tangentlines}, because no further cancellation may happen. If $r_1>0,$ the ray at heights $(r,\infty)$ in the cone picture continues down to the height $r_1,$ where two things may happen. Either there is no cancellation and we get an expression of the type (\ref{eq_tangent_r}), or there is a new cancellation at the top order, and we need to go to a deeper asymptotic expansion. In the latter case, we would have again proportionality of the top matrix to $A_0,$ meaning that the ray continues in a straight way. The procedure will terminate eventually by virtue of working with Puiseux series. 

\begin{thm}
Let $L$ be a line in $\mathbb{K}P^3$ intersecting $Q(\mathbb{K})$in two points. Then $\VAL(L)$ is either \begin{equation}
\{r\}\times\sigma(l\backslash\{p,q\}) \cup (r,\infty)\times\mathcal{S}|_{\{p,q\}}\cup\{\infty\}\times\{p,q\},
    \label{eq_generalline_r}
\end{equation}
where $r>0$, $l$ is a line on $Q(\mathbb{C}),$ $p$ and $q$ are distinct points on $l$ and $\sigma$ is a section of $\mathcal{S}|_{l\backslash\{p,q\}},$ or
\begin{equation}
\{0\}\times\varkappa^\circ(l\backslash \{p,q\})\cup (0,\infty)\times\mathcal{S}|_{\{p,q\}}\cup\{\infty\}\times\{p,q\},
    \label{eq_generalline_0}
\end{equation}
where $l$ is a line in $\mathbb{K}P^3$ intersecting $Q(\mathbb{C})$ at two points $p$ and $q,$ or $\VAL(L)$ is of the form (\ref{eq_tangent_r}) or (\ref{eq_tangent_0}).
    \label{thm_generalline}
\end{thm}

The only type of lines we haven't considered yet is that of $L$ belonging entirely to $Q(\mathbb{K}).$ One way to deal with it is to think of such a line as being parametrized by $[v_1 \cdot v_2^\top]_{\mathbb{K}^*},$ where $v_1$ and $v_2$ are two nonzero vectors, one of which is fixed and the other varies freely. The corresponding image under $\VAL$ will be a line of the same type but now on $Q(\mathbb{C})$ simply by taking the top degree part of the fixed vector -- the top degree part of the second vector will again vary freely.

As a final remark for this section, we emphasize that the computations above should be doable whenever there is an explicit parametrization of a curve, such as when the curve is rational. It will certainly be harder, the case of conics already looks intimidating. However, we expect to give a complete description of $\VAL$ images for all rational curves. Higher genus case, on the other hand, at the moment looks virtually impossible to handle in full generality.
\vspace{-5pt}
\section{The case of surfaces}
In this section, we aim to find restrictions on the image under $\VAL$ for some surfaces in every degree as a preparation for the following section where we will see that no line belongs to any of them (assuming that the input complex polynomials $f_k$ are generic) starting with degree $4$. By the nature of the setup, there is a slight variation between even and odd degrees reflected in the odd degree by the presence of a spherical coamoeba of a generic plane over the tip of the cone. At this moment, we know what is its shape -- a complement to a cylinder; this fact, however, is not relevant to the present article and will be proven elsewhere.

The image under $\VAL$ will be decomposed into the following parts: $\Sigma_0$ -- a part over the tip of the cone, which will be empty in our examples of even degree; $\Sigma_\infty$ -- a part at the infinite level, i.e. at the base of the cone, consisting of the intersection of a hypersurface and the quadric; $\Sigma_R$ -- the biggest part consisting of regular values, each of its slice at a given height will be consisting of the total space of the circle bundle $\mathcal{S}$ restricted to a symmetric bidegree curve on the quadric; $\Sigma_C$ -- the union of slices at critical levels each of which is an image of a section of $\mathcal{S}$ defined in the complement of a union of the two curves to the sides of the critical level, together with the whole fiber at their intersection.

A (phase) tropicalization can be defined as the closure of an image under (phase) valuation of a variety -- note that the images under the valuation and coamoeba map generally are not closed, in contrast with usual (Archimedean) amoebas which are always closed. We believe that the examples described below are actually typical, i.e. a very specific choice of the form of a polynomial is not crucial -- allowing perhaps, obvious degeneration of critical levels. Moreover, analogously to \cite{kerr2018phase} we, expect that in the generic case the corresponding closure is a topological manifold homeomorphic to a smooth complex surface of the corresponding degree.

Here, on the other hand, we emphasize that we are not going to take the closure, keeping the gaps at the attachment of regular and critical levels. This will be crucial in observing no lines in the next section. 

\begin{thm}
    Let $n$ be a natural number and $S_{2n}$ be a surface in $\mathbb{K}P^3$ of degree $2n$ given by homogeneous equation $F_{2n}=0,$ where $$F_{2n}(A)=\sum_{j=0}^n t^{-j(j+1)}(\det(A))^{n-j}f_{2j}(A),$$ where $f_{2j}$ is a degree $2j$ homogeneous polynomial with complex coefficients not divisible by $\det$ considered as a degree $2$ polynomial. Denote by $C_{2j}$ a complex curve on $Q(\mathbb{C})$ defined by the restriction of $f_{2j}.$  Then $$\VAL(S_{2n})\subset \Sigma_\infty\cup\Sigma_R\cup\Sigma_C,$$ where $\Sigma_\infty=\{\infty\}\times C_{2n},$ $\Sigma_R=\cup_{j=1}^{n-1}(j,j+1)\times\mathcal{S}|_{C_{2j}}\cup (n,\infty)\times\mathcal{S}|_{C_{2n}}$ and $\Sigma_C=\cup_{j=1}^n\{j\}\times(\sigma_j(Q(\mathbb{C})\backslash (C_{2j-2}\cup C_{2j}))\cup \mathcal{S}|_{C_{2j-2}\cap C_{2j}})$ and $\sigma_j[B]_{\mathbb{C}^*}=[\sqrt{-\frac{f_{2j-2}(B)}{f_{2j}(B)}}B]_{\mathbb{R}^*}$ is a section of $\mathcal{S}|_{Q(\mathbb{C})\backslash (C_{2j-2}\cup C_{2j})}.$
    \label{thm_evendegree}
\end{thm}

\begin{proof}
Take $A=t^\beta B+o(t^\beta)$ such that $F_{2n}(A)=0.$ Assume first $\det(A)=0.$ This immediately implies $f_{2n}(A)=0.$ From the last two equations we get two conditions on $B:$\ $\det(B)=0$ and $f_{2n}(B)=0.$ Therefore, $\VAL(A)\in\Sigma_\infty.$

If $A$ is invertible, normalize it to have $\det(A)=1.$ Observe again that this condition implies $\beta\geq 0.$ However, if $\beta<1,$ the first term  $\det(A)^nf_0(A)$ will dominate all others. This term, is, however, nonzero, therefore such a situation is not possible. In particular, $\VAL(S_{2n})$ has no points at heights in $[0,1).$  

More generally, substitution in $F_{2n}(A)$ of all determinants with $1,$ results in an inhomogeneous polynomial that we further expand as a series in $t.$ 
The expected valuation of a top order term in $F_{2n}(A)$ is given by the tropical polynomial $\max_{j=0}^n(-j(j+1)+2j\beta),$ which has roots at $1,\dots,n.$ If $\beta$ is not one of these roots, $F_{2n}(A)$ has exactly one top order term $t^{-j(j+1)+2j\beta}f_{2j}(B)$ which must vanish. This implies that $[B]_{\mathbb{C}^*}$ belongs to $C_m$ and $\VAL(A)\in \Sigma_R.$   

In the situation when $j=\beta\in\{1,\dots,n\},$ there are two competing terms in $F_{2n}(A),$ the sum of which must vanish. We get an inhomogeneous equation $f_{2j-2}(B)+f_{2j}(B)=0$ on $B.$ If both $f_{2j-2}(B)$ and $f_{2j}(B)$ do not vanish (i.e. if $[B]_{\mathbb{C}^*}\in Q(\mathbb{C})\backslash (C_{2j-2}\cup C_{2j})$), this equation has a solution $cB\in[B]_{\mathbb{C}^*}$ for some $c\in\mathbb{C}^*.$ Explicitly, by substituting $cB$ into the inhomogeneous equation we get $c=\sqrt{-\frac{f_{2j-2}(B)}{f_{2j}(B)}}$ and $[B]_{\mathbb{C}^*}\mapsto [cB]_{\mathbb{R}^*}$ defines the section $\sigma_j$ in the statement. That gives a point of $\Sigma_C.$

Finally, the case when one of $f_{2j-2}(B)$ or $f_{2j}(B)$ vanishes but the other does not, contradicts the assumption $F_{2n}(A)=0.$ If both vanish, $[B]_{\mathbb{C}^*}\in C_{2j}\cup C_{2j-2}$ and we do not put any additional restriction on $B,$ landing under $\VAL$ to the fiber of $\mathcal{S}$ over $[B]_{\mathbb{C}^*}$
\end{proof}
    
\begin{thm}
    Let $n$ be a natural number and $S_{2n+1}$ be a surface in $\mathbb{K}P^3$ of degree $2n+1$ given by homogeneous equation $F_{2n+1}=0,$ where $$F_{2n+1}(A)=\sum_{j=0}^n t^{-j(j+1)}(\det(A))^{n-j}f_{2j+1}(A),$$ where $f_{2j+1}$ is a degree $2j+1$ homogeneous polynomial with complex coefficients not divisible by $\det$ considered as a degree $2$ polynomial. Denote by $C_{2j+1}$ a complex curve on $Q(\mathbb{C})$ defined by the restriction of $f_{2j+1}.$  Then $$\VAL(S_{2n+1})\subset \Sigma_\infty\cup\Sigma_R\cup\Sigma_C\cup\Sigma_0,$$ where $\Sigma_\infty=\{\infty\}\times C_{2n+1},$ $\Sigma_R=\cup_{j=0}^{n-1}(j,j+1)\times\mathcal{S}|_{C_{2j+1}}\cup (n,\infty)\times\mathcal{S}|_{C_{2n+1}}$ and $\Sigma_C=\cup_{j=1}^n\{j\}\times(\sigma_j(Q(\mathbb{C})\backslash (C_{2j-1}\cup C_{2j+1}))\cup \mathcal{S}|_{C_{2j-1}\cap C_{2j+1}})$ and $\sigma_j[B]_{\mathbb{C}^*}=[\sqrt{-\frac{f_{2j-1}(B)}{f_{2j+1}(B)}}B]_{\mathbb{R}^*}$ is a section of $\mathcal{S}|_{Q(\mathbb{C})\backslash (C_{2j-1}\cup C_{2j+1})},$ and $\Sigma_0=\{0\}\times \varkappa^\circ(H),$ where $H$ is the intersection with $PSL_2(\mathbb{C})$ of a plane in $\mathbb{C}P^3$ defined by $f_1=0.$
    
    \label{thm_odddegree}
\end{thm}

\begin{proof}
    The argument is quite analogous to the even degree case, except for the situation when $\beta=0$ in the asymptotic expansion $A=Bt^\beta+o(t^\beta).$ Under the assumption $\det(A)=1$ this implies $\det(B)=1$ and by substituting such $A$ to $F_{2n+1},$ its leading term is $f_1(B)(\det(A))^n=f_1(B)$ the vanishing of which defines $H,$ an intersection of a plane with $PSL_2(\mathbb{C}).$ Clearly, $\VAL(A)\in\Sigma_0.$
\end{proof}

We expect that the inclusions in Theorems \ref{thm_evendegree} and \ref{thm_odddegree} are actually equalities. Moreover, allowing critical levels to vary and be multiple (in which case the sections would be multi-sections), as well as permitting for higher degree coamoebas above the tip of the cone in both even and odd cases, the description of the $\VAL$ images of surfaces is likely to be complete. 

\section{Lines on surfaces}
In what follows, we combine the results of the previous two sections to illustrate how accurate is $PSL_2$ phase tropicalization in describing the relation between generic surfaces and lines in $\mathbb{P}^3.$ 
As it was explained, the ordinary phase-forgetting $PSL_2$ tropicalization is hopeless in capturing this relation, but recovering the phase helps to restore the most characteristic elements of the picture.

First, observe that $\VAL(S_1)$ admits a two-dimensional family of $\VAL$-images of lines. If $f_1$ is generic, then $C_1$ is an irreducible $(1,1)$-curve, in which case only types (\ref{eq_tangent_0}) and (\ref{eq_generalline_0}) can fit to $\VAL(S_1),$ since (\ref{eq_tangent_r}) and (\ref{eq_generalline_r}) have a section over either $(1,0)$-curve or $(0,1)$-curve at a non-zero level. The first type is parametrized by a single (double) point on $C_1$ and the second is parametrized by two distinct points on the same curve.  
Second, $\VAL(S_2)$ has one critical level at height $1,$ being empty for heights less than one, which implies that types (\ref{eq_tangent_0}) and (\ref{eq_generalline_0}) cannot belong to it. After height 1, it is spanned by the total space of $\mathcal{S}$ over the curve $C_2.$ Assuming $C_2$ is irreducible, prohibits the types \ref{eq_tangent_r} and \ref{eq_generalline_r} if $r>1.$ Thus, $r$ can be only equal to $1$ and we get a one-dimensional family of lines consisting of two disjoint components.

Third, $\VAL(S_3)$ also has a single critical level. We will show now that every type of a $\VAL$ image of a line is rigid on $\VAL(S_3)$ provided that $f_1$ and $f_3$ are generic. The types (\ref{eq_tangent_0}) and (\ref{eq_generalline_0}) are rigid since they cannot be pushed from the tip of the cone by the irreducibility of $C_1$ and their rays must be based at the finite intersection of $C_1$ and $C_3.$ For the types \ref{eq_tangent_r} and \ref{eq_generalline_r} again $r$ must be equal to $1.$ Moreover, at each intersection of the line $l$ with $C_3\backslash C_1$ there must be a gap (where the section $\sigma_1$ is not defined) where a ray is attached. Note that in general, there are three such intersections of $l$ and $C_3$, but we may have at most two rays on the image of $\VAL$ for a line. Thus, the only possibility is that $l$ passes through one of the $6$ intersection points of $C_1$ and $C_3,$ which together with the choice of the bidegree (either $(0,1)$ or $(1,0)$) completely determines $l.$

As for higher degrees of surfaces, we have the following.

\begin{thm}
    Let $d$ be a natural number greater than $3$ and let $S_d$ be a degree $d$ surface defined in either in Theorem \ref{thm_evendegree} or \ref{thm_odddegree} such that all the polynomials $f_j$ are generic. Then, there exists no line $L$ on $\mathbb{K}P^3$ such that $\VAL(L)\subset\VAL(S_d).$  
    \label{thm_noline}
\end{thm}

\begin{proof}
Let's start with types (\ref{eq_tangent_0}) and (\ref{eq_generalline_0}). Note that $\VAL(S_d)$ has at least two critical levels, and the rays in the images of lines would need to penetrate through them, passing through the intersection of three generic curves on $Q(\mathbb{C}),$ which is empty.

As for other types, their critical level $r$ must be at the critical level of $\VAL(S_d)$ since all the curves $C_j$ are irreducible by the genericity assumption. By the previous argument, $r$ can be only $n$ or $n-1,$ where $n$ is the rounding down of $\frac{d}{2},$ because the rays should penetrate through all the corresponding regular levels.

Assume first $r=n.$ Then, $l$ must have at most two intersection points with $(C_d\cup C_{d-2})\backslash (C_d\cap C_{d-2})$ where the rays could be attached. However, by genericity of $C_d$ its projection to one of the factors of $Q(\mathbb{C})\cong\mathbb{C}P^1\times\mathbb{C}P^1$ has only a simple ramification profile of order two, with ramification points having different images. Moreover, by genericity of $C_{d-2},$ this ramification points are distinct from the images of the intersection of $C_{d-2}$ with $C_d.$ Thus, a line $l$ may have an intersection point of multiplicity two with $C_d,$ but then all other intersections are transverse (with their total number equal $d-2$) and are all away from $C_{d-2},$ giving at least $3$ intersections of $l$ with $(C_d\cup C_{d-2})\backslash (C_d\cap C_{d-2}).$

Finally, assume $r=n-1.$ Then, the rays must penetrate through two regular levels and their endpoints must be at the intersection of two curves $C_d$ and $C_{d-2}.$ Moreover, neither of these intersections belongs to the same line $l$ by genericity. Thus $\VAL(L),$ if it belongs to $\VAL(S_d)$ must have a single ray, and $l$ must pass through its endpoint. On the other hand, again by genericity, $l$ will not pass through $C_{d-2}\cap C_{d-4},$ so it will have $d-2-1\geq 1$ additional gaps since all the intersections of $l$ with $C_{d-2}$ are transverse. A contradiction.
\end{proof}

Here we once more focus attention on the importance of not taking the closure of $\VAL(S_d),$ since otherwise we would lose the gap argument we used heavily above. Observe that in the geometric version of tropicalization, where we would take the Hausdorff (or Kuratowski) limit of a family, the result would be automatically close. It is a curious problem if the gaps could be restored in this setup.  

\section{Perspectives}
We emphasize that our point is not reproving the well known fact of no-lines on generic degree$\geq 4$ surfaces, but rather showing the consistency of $PSL_2$ phase tropical geometry with the classical complex geometry at this basic level, something which is not there if the phase would be forgotten (the difference is less dramatic in the case of the usual tropicalization). Starting in that way there is a possibility for further exploration and adjustment. Notably, it might be helpful to further extend our tropicalization procedure to solve more elaborate questions, such as enumerative problems. %For example, we are not yet able to reproduce 27 lines on a cubic surface -- the naive count gives six too many. We suspect that the excess lies in non-realizability of lines of the type \ref{eq_tangent_0} on a $PSL_2$ tropicalization of a generic cubic.

One such possibility for further refinement would be redefining $\VAL$ as $t^\alpha B+o(t^\alpha)\rightarrow (\alpha,[B]_{\{\pm 1\}})$ which would mean that instead of circle bundle $\mathcal{S}$ we would be dealing with a $\mathbb{C}^*$-bundle in a way more natural for complex geometry. It also might happen that for particular tasks forgetting the bundle altogether is enough, which is certainly true for the purposes of basic intersection theory. Such potential adjustments are aligned with the question posed by Oleg Viro. Namely, what kind of hyperfield-type structure on matrices produces shapes appearing in Theorems \ref{thm_evendegree} and \ref{thm_odddegree}? 

Of course, there is still much that remains unexplored. Remarks and suggestions would be much welcome. As was mentioned, we expect that the inclusions of Theorems \ref{thm_evendegree} and \ref{thm_odddegree} above are actually equalities. %Perhaps, this can be solved by purely algebraic manipulations with equations and series. However, we are looking for a more satisfactory explanation.%One could come from the above mentioned intersection theory -- if some point from the expected tropicalization of a generic surface is missing we would get fewer intersections with a tropicalization of some line. This heuristic would require a type of lifting argument, which could be achieved by unifying the algebraic and geometric approaches to our tropicalization procedure.

%As for our motivation for doing all this, the dream is to have a convenient tool for observing phenomena in low-dimensional geometry. As was mentioned, $\mathbb{C}P^3$ is the natural compactification of $PSL_(\mathbb{C}),$ the only one having the properties of being smooth and equivariant. Thus, more at the aesthetic level, we hope that the approach presented in this paper eventually would lead to a better understanding of the projective three-space. Going to groups of higher rank would require advanced representation theory and lead to many compactifications, it is not the principal direction of our current research. What is, on the other hand, is extending the setup to $\mathbb{C}P^4,$ which is quite straightforward at the formal level. As we demonstrated above, our technique under construction, perhaps yet insufficient for precisely counting objects, is good enough for showing that something doesn't exist and estimating dimensions. A particular question we have in mind is that if a generic quintic threefold has only finitely many rational curves in every degree, which is the formidable Clemens' conjecture.

\end{document}